\documentclass[a4paper,11pt]{amsart}
\usepackage[utf8]{inputenc}
\usepackage{indentfirst, amsfonts, amsmath, amsthm, amssymb, amscd}
\usepackage{amsmath,amsfonts,amscd,bezier}
\usepackage{graphicx}
\usepackage{color}
\usepackage[all]{xy}
\usepackage{comment}
\usepackage{enumitem}

\usepackage{hyperref}
\newtheorem{theorem}{Theorem}[section]

\newtheorem{proposition}[theorem]{Proposition}

\newtheorem{definition}[theorem]{Definition}

\theoremstyle{remark}

\theoremstyle{plain}
\newtheorem{main}{Theorem}

\newcommand{\ds}{\displaystyle}

\newcommand{\R}{\mathbb{R}}
\newcommand{\Z}{\mathbb{Z}}
\newcommand{\T}{\mathbb{T}}

\newcommand{\cF}{\mathcal{F}}

\newcommand{\cA}{\mathcal{A}}

\newcommand{\mS}{\mathcal{S}}
\newcommand{\OO}{\mathcal{O}}

\title{Anosov actions: minimality of foliations or suspension action}
\author{R.R. Lopes}
\address{DAMAT,
  UTFPR, Pato Branco-PR, Brazil.}
  \email{rodrigorlopes@utfpr.edu.br}
	\author{C. Maquera}
	\address{Departamento de Matem\'atica,
  ICMC - USP, S\~ao Carlos-SP, Brazil.
	}
\email{cmaquera@icmc.usp.br}
\author{R. Var\~{a}o} 
\address{Departamento de Matem\'atica, Estat\'istica e Computa\c c\~ao Cient\'ifica,
  IMECC-UNICAMP, Campinas-SP, Brazil.}
\email{varao@unicamp.br}

\begin{document}

\maketitle

\begin{abstract}
We prove that an Anosov action of $\mathbb R^k$ over a compact manifold $M$ transitive on regular sub-cones satisfies the dichotomy: each stable and unstable leaf is dense or the Anosov action is topologically conjugated to a suspension of a $\mathbb Z^k$-Anosov action. This represents an important progress toward addressing Verjovsky’s extended conjecture for Anosov actions, as developed by Barbot and Maquera.
\end{abstract}

% \tableofcontents

\section{Introduction}
Dmitri V. Anosov, in his classical monograph \cite{anosov-monograph}, proved what is now known as the \textit{Anosov alternative} for flows, which states that every strong unstable (or strong stable) leaf is dense or the Anosov flow is a suspension over an Anosov diffeomorphism. In the seventies, Verjovsky conjectured that every codimension one Anosov flow on a manifold of dimension $\geq 4$ is topologically equivalent to a suspension over a toral Anosov automorphism. Later, Barbot and Maquera \cite{barbot-11, maquera-11} have extended the Verjovsky conjecture for $\mathbb R^k$-actions.

\textbf{Verjovsky's conjecture for $\mathbb R^k$-actions:} Every irreducible codimension one Anosov action of $\mathbb R^k$ on a manifold of dimension at least $k+3$ is topologically equivalent to the suspension of a linear Anosov action of $\mathbb Z^k$ on the torus.

As mentioned by Barbot and Maquera \cite{barbot-11}, the most likely idea to prove Verjovsky's Conjecture seems to be to find a global cross section
to the $\mathbb{R}^k$-Anosov action, which would imply the existence of fibered space of the manifold $M$ over the torus $\mathbb{T}^k$. The proof of our main result, stated below, follows these ideas.

\begin{main}\label{th:dichotomy}
Let $\phi: \R^k\times M \rightarrow M$ be a transitive Anosov action by regular subcones. Then one of the following occurs:
\begin{enumerate}
\item Each strong unstable and each strong stable leaf is dense on  $M$, or;
\item The action $\phi$ is topologically conjugated to a suspension of a $\Z^k$-Anosov action on a compact manifold of dimension $\dim M - k$.
\end{enumerate} 
\end{main}

We generalize the \textit{Anosov alternative} for flows, following \cite{anosov-monograph,plante}, to Anosov actions. We also point out that Theorem \ref{th:dichotomy} is a partial answer to the Barbot-Maquera extension of the Verjovsky conjecture for $\mathbb R^k$-actions since, in the context of our main result above, if one of the foliations $\mathcal{F}^{ss}$ or $\mathcal{F}^{uu}$ is not minimal, then the action is a suspension of a $\mathbb{Z}^k$ action.

\section{Preliminaries}\label{sec:preliminaries}
\label{prerequisitos}

In this section, we review a few basic definitions and present the results needed throughout the paper. 

\subsection{Lie group actions}
An action of a Lie group $G$ in a manifold $M$ is a $C^r$ map $\phi:G \times M \rightarrow M$ (with $(r\geq 1)$)  that satisfies the following conditions: 

\begin{enumerate}
 \item $\phi(e,p)= p$, for all $p\in M$,  where $e$ denotes the identity element of $G$;
 \item $\phi(u\cdot v,p)=\phi(u,\phi(v,p))$, for all $u,v \in G$ and for all $p \in M$.
\end{enumerate}

 For each $a\in G$, the action induces a diffeomorphism
$\phi^a$ defined by $\phi^a(p)=\phi(a, p)$ for all $p\in M$. 
The orbit of $p\in M$ under the action $\phi$ is defined as $$\OO_p = \OO_p(\phi)=\left\{\phi^a(p)\ |\  a\in G \right\}.$$ 

The isotropy subgroup of $p\in M$ with respect to the action $\phi$ is given by  $$\Gamma_p= \left\{v\in G\ |\ \phi^v(p)=p\right\}.$$ 
It is immediate that if $q=\phi^u(p)$, then $\Gamma_q=\Gamma_p$. Thus, the isotropy subgroup of as orbit $\mathcal{O}_p$ can be denoted by $\Gamma_{\mathcal{O}_p}$. 

If the action is $C^r$, then for each $p\in M$ the map $\phi_p: G \rightarrow M$, defined by $\phi_p(v) = \phi(v, p)$, is also $C^r$ and induces the map $\overline{\phi}_p:G/\Gamma_p\rightarrow M$ which is an injective immersion whose image is $\OO_p$. Consequently, the isotropy subgroup $\Gamma_p$ is closed. Therefore, topologically, the orbit $\OO_p$ has as many possibilities as the closed subgroups of $G$. In particular, if $G = \mathbb R^k$ then each orbit can be homeomorphic to $\R^k$ or $\T^k$ or $\R^{k-l}\times \T^l$ with $1\leq l \leq k-1$. 

An action $\phi: G \times M \rightarrow  M$ is called a foliated action if, for every $p \in M$, the tangent space $T_p \OO_p(\phi)$ has fixed dimension $l$. In this case, the subbundle of the tangent bundle $TM$ formed by the tangent spaces to the orbits defines a foliation on $M$, denoted by $T\phi$. If $l = \dim G$, the action $\phi$ is locally free.  The orbits of the foliated action $\phi$ correspond to the leaves of the foliation on the manifold $M$.

 Consider an action $\psi$ of $\mathbb{Z}^k$ in a closed manifold $S$. The suspension of $\psi$ is the quotient space $M$ obtained from $S\times \mathbb{R}^k$ by identifying each $(x,u)$ with $(\psi^v(x),u+v)$ for every $v$ in $\mathbb{Z}^k$, and is equipped with the action $\phi: \mathbb{R}^k\times M \rightarrow M$ defined by $\phi(v,[(x,t)])=[(x,t+v)]$. By construction, the canonical projection $\pi: S\times \mathbb{R}^k\rightarrow M$ is a covering map and guarantees that the orbit of $\phi$ has dimension $k$. Additionally, for all $v\in \mathbb{R}^k$,  $S_v=\pi(S\times\{v\})$ is an embedded submanifold in $M$, diffeomorphic to $S$, which defines a transversal foliation to the $\phi$-orbits. In particular,  if $\{e_i\}$ is the canonical base of $\mathbb{R}^k$ then $\pi((x,e_i))=\pi(\psi^{e_i}(x),0)$, thereby determining a ``first return''  on $S_0$ of the action $\phi$ in the direction of the vector $e_i$.

\subsection{Anosov actions}

Hirsch, Pugh, and Shub have defined Anosov actions and studied them in general context \cite{hirsch,pugh-shub}. 

\begin{definition} \label{defanosov}
An action $\phi$ of $\mathbb R^k$ on a smooth manifold $M$ is called an Anosov action if it is locally free and there exists an element $a\in\mathbb{R}^k$ such that the diffeomorphism $g=\phi^{a}$ has the following properties: There exist real constants $\lambda =\lambda(a) > 0$, $C= C(a)>0$ and a splitting of  $TM$ invariant by $Dg$, $TM=E_{a}^{ss}\oplus T\phi \oplus E_{a}^{uu}$, such that:
\end{definition}
\begin{enumerate}
 \item $||Dg^n|_{E_{a}^{ss}}||\leq Ce^{-\lambda n},\ \forall n>0$;
 \item $||Dg^n|_{E_{a}^{uu}}||\leq Ce^{\lambda n}, \ \ \forall n<0$.
\end{enumerate}

The element $a\in \R^k$ from the definition above is known as the \emph{hyperbolic element} or \emph{Anosov element} of the action. We denote by $\cA=\cA(\phi)$ the set of Anosov elements of the action $\phi$. An Anosov action is said to be codimension one if the subbundle $E_a^{ss}$ or $E_a^{uu}$ is one-dimensional. According to the Hisch-Pugh-Shub theory \cite{hirsch} for normally hyperbolic maps, the set of Anosov elements $\cA$ is an open set in $\R^k$, and each connected component of $\cA$ is a convex open cone. The subbundles $E_{a}^{ss}$ and $E_{a}^{uu}$ are invariant under the action $\phi$ because for all $v\in \R^k$, these subbundles are $D\phi^v$-invariant. In fact, if $h: M\rightarrow M$ is a diffeomorphism such that the subbundle $T\phi$ is $Dh$-invariant, then the subbundles $E_{a}^{ss}$ and $E_{a}^{uu}$ are also $Dh$-invariant. Therefore, the constants $C= C(a)$ and $\lambda = \lambda(a)$ can be fixed in each connected component. Thus, the subbundles $E_{a}^{ss}$ and $E_{a}^{uu}$ can remain the same for a connect component $\cA_0$ of $\cA$. Consequently, we consider the dynamics of the action $\phi$ on $\cA_0$ where the subbundles $E^{ss}$ and $E^{uu}$ are fixed.

\begin{definition}\label{defsubregular}
A regular subcone $\mS$ is an open connected subcone contained in $\cA_0$. Given $x \in M$, the orbit of $x$ with respect to the regular subcone $\mS$ is the set $\OO_{\mS}(x) = \{\phi^v(x) \ | \ v \in \mS\}$. 
\end{definition}

Hirsch, Pugh and Shub \cite{hirsch} proved the H\"older continuity of the subbundles $E^{ss}$, $E^{uu}$, $E^{ss}\oplus T\phi$ and $E^{uu}\oplus T\phi$, as well as their integrability. The corresponding foliations $\mathcal{F}^{ss}$, $\mathcal{F}^{uu}$, $\mathcal{F}^{s}$ and $\mathcal{F}^{u}$ are  respectively called the strong stable foliation, strong unstable foliation, weak stable foliation, and weak unstable foliation. It is also known that these foliations are absolutely continuous and for a given point $p\in M$, the weak leaves are
\begin{align} \label{formula1}
	\mathcal{F}^{s}(p) = \displaystyle \bigcup_{q\in \mathcal{O}_p}\mathcal{F}^{ss}(q)\ \ \mathrm{e} \ \ 
	\mathcal{F}^{u}(p) = \displaystyle \bigcup_{q\in \mathcal{O}_p}\mathcal{F}^{uu}(q).
\end{align}

The metric on the leaf  $\cF^i(x)$ is the induced metric from the manifold  $M$, where $i = s,\ ss,\ u,\ uu$. For an adapted metric in $M$, it can be assumed $C=1$ and $e^\lambda<1$. Consequently, for each $x\in M$ and $a\in S$: 

 \begin{itemize}
     \item  If $y,z\in \cF^{ss}(x)$, then for $t>0$,
    $$d(\phi^{ta}(y),\phi^{ta}(z)\leq e^{\lambda t} d(y,z).$$ 
    \item If $y,z\in \cF^{uu}(x)$, then for $t>0$,
    $$ d(\phi^{-ta}(y),\phi^{-ta}(z)\leq e^{\lambda t} d(y,z).$$
    \item If $y,z \in \cF^{s}(x)$, then for $t>0$, 
$$ d(\phi^{ta}(y),\phi^{ta}(z)\leq d(y,z).$$
\item If $y,z\in \cF^u(x)$, for $t>0$, then $$d(\phi^{-ta}(y),\phi^{-ta}(z)\leq  d(y,z).$$
 \end{itemize}
 
For $\delta> 0$, the set $\cF^i_\delta(x)$ is the ball centered at $x$ with radius $\delta$ contained  in the leaf $\cF^i(x)$. The topology of $M$ and the foliations are related, as there exist $\gamma> 0$ such that for $0<\delta <\gamma$ the ball $\cF^i_\delta(x)$ is given by $\cF^i_\gamma(x) \cap B_\delta(x)$, where $B_\delta(x)$ the ball centered at $x$ with radius $\delta$ contained in $M$.

Note that the foliations $\cF^s$ and $\cF^{uu}$ are transversal to each other, as well as $\cF^{ss}$ and $\cF^u$. One can find local charts of $M$ that simultaneously trivialize  $\cF^s$ and $\cF^{uu}$.  Likewise, there are local charts that simultaneously trivialize $\cF^{ss}$ and $\cF^u$. These charts induce a local product structure on $M$, as described by the following theorem.

\begin{theorem}[Local product structure \cite{maquera-11}]\label{produtolocal2}
Let $\phi:\mathbb{R}^k\times M \rightarrow M$ be an Anosov action on a compact connected manifold $M$. Then there exists $\delta_0>0$ such that for all $\delta \in (0,\delta_0)$ and for all $x\in M$, the maps
$$\begin{array}{c}
[\cdot,\cdot]^{u} : \mathcal{F}^{s}_\delta(x) \times \mathcal{F}^{uu}_\delta(x) \rightarrow M;\ \ \ [y,z]^{u} = \mathcal{F}^{s}_{2\delta}(z) \cap \mathcal{F}^{uu}_{2\delta}(y)\\

[\cdot,\cdot]^{s} : \mathcal{F}^{ss}_\delta(x) \times \mathcal{F}^{u}_\delta(x) \rightarrow M;\ \ \  [y,z]^{s} = \mathcal{F}^{ss}_{2\delta}(z) \cap \mathcal{F}^{u}_{2\delta}(y)
\end{array}$$ 
are homeomorphisms over their images.
\end{theorem}

\subsection*{Non-wandering set}
As in discrete dynamical systems, the non-wandering set plays a vital role in group action dynamics. To understand the non-wandering set of an Anosov action, we first need to study the action's periodic elements, including the relationship between isotropy subgroups and compact orbits. 

The following results demonstrate that all the information about the dynamics of an Anosov action can be found within the regular subcones. We begin with the Closing Lemma, an essential tool for dynamical systems induced by group actions.

\begin{theorem}[Katok - Spatzier \cite{katok-94}]\label{closing} If $a\in \cA_0$ then there exist positive constants $\varepsilon_0$, $C$ and $\lambda$ depending continuously on $\phi$ in the $C^1$ topology such that for all $x \in M$ and $t>0$ satisfying $$d(\phi^{ta}x, x) < \varepsilon_0,$$ then there exists a point $y, \in M$ and a differentiable curve $\gamma: [0,t] \rightarrow \R^k$ such that for all $s\in[0,t]$ one has:
\end{theorem}
\begin{enumerate}
\item $d(\phi^{sa}x, \phi^{\gamma(s)}y) < C e^{-\lambda \min\{s,t-s\}} d(\phi^{ta}x,x)$;
\item $\phi^{\gamma(t)}y = \phi^{\delta} y$, for some $\delta$ with $||\delta ||< C d(\phi^{ta}x,x)$;
\item $||\gamma' - a || < C d(\phi^{ta}x,x)$. 
\end{enumerate}

Note that the point $y$ is a fixed point for $\phi^{\gamma(t)-\delta}$, which means that the orbit $\OO_y$ is a cylinder or compact.

\begin{proposition}[Barbot - Maquera \cite{maquera-11}] \label{orbcompact} Let $\phi$ be an Anosov action of $\R^k$ on a smooth connected manifold $M$. Then each orbit action $\OO_p$, whose isotropy subgroup $\Gamma_p$ contains an Anosov element $a \in\cA_0$, is compact. 
\end{proposition}

Every orbit is topologically equivalent to $\R^k$, $\T^k$ or a cylinder of the form $\R^{k-l}\times\T^l$ with $1\leq l \leq k-1$. The above proposition ensures that the third case does not occur because $\cA_0$ is a convex open cone, and every isotropy subgroup is a lattice. In Lie group theory, a lattice is a discrete subgroup $\Lambda$ of a Lie group $G$ such that the quotient $G/\Lambda$ is compact. Essentially, a lattice in a finite-dimensional real vector space $E$ is a subset $\Lambda \subset E$ such that there exists a basis $\{v_1,\ldots, v_r\}$ of $E$ for which $\Lambda$ consists of all integer linear combinations of $v_1,\ldots, v_r$.

\begin{proposition}[Barbot - Maquera \cite{maquera-11}]\label{reticuladoAnosov}
Let $\mS$ be a regular subcone of $\cA_ 0$. Then, every lattice of $\R^k$ contains Anosov elements in $\mS$. 
\end{proposition}

In particular, it is possible to find a base of $\R^k$ whose elements belong to a specific isotropy subgroup of interest. To take advantage of the properties of $\mS$ along with the compact orbits, we introduce the concept of periodic points associated with $\mS$. 

\begin{definition}[The set of periodic points]
A point $x\in M$ is called a periodic point with respect to the regular subcone  $\mS$ if there exists $v \in \mS$ such that $x$ is a fixed point for  $\phi^v$.  The set of all such periodic points is denoted by  $\mathrm{Per}(\mS)$. 
\end{definition}

On the one hand, if  $x\in M$ is a periodic point with respect to $\mS$, then by Proposition \ref{orbcompact}, the orbit $\OO_x$ is compact. This implies that $\mathrm{Per}(\mS)$ is contained within the union of compact orbits. On the other hand, since the isotropy subgroup $\Gamma_p$ is a lattice by Proposition \ref{reticuladoAnosov}, the orbit compact $\OO_p$ is contained in $\mathrm{Per}(\mS)$. Therefore, the set $\mathrm{Per}(\mS)$ is precisely the union of all compact orbits of the Anosov action for any regular subcone $\mS$. 

\begin{definition}[Non-wandering set]\label{defconjnaoerr} A point $p \in M$is called non-wandering with respect to a regular subcone $\mS$ if, for any open set $U$ containing $p$ there exists $v\in \mS$, $||v|| > 1$, such that $\phi^vU \cap U \neq \emptyset$. The set of all non-wandering points is denoted by $\Omega(\mS)$.
\end{definition}

The non-wandering set is nonempty because $\mathrm{Per}(\mS)$ is contained within $\Omega(\mS)$. Note that if $x\in \Omega(\mS)$, there exist $a\in \mS$ such that $d(\phi^{ta}x,x)$ can be made arbitrarily small. By Theorem \ref{closing}, this implies that the derivative $\gamma'$ and $\gamma(t) - \delta$ are contained in $\mS$. In particular, the orbit $\OO_y$ of $y$  as considered in Theorem \ref{closing}, is compact by Proposition \ref{orbcompact}. Hence, $\mathrm{Per}(\mS)$ is dense in $\Omega(\mS)$. 

\begin{proposition}[Barbot - Maquera \cite{maquera-11}]	\label{orbitacompacdensa}
Let $\phi$ be an Anosov action of $\R^k$ on $M$, then for every regular subcone $\mS$, the union of compact orbits of $\phi$ is dense in  $\Omega(\mS)$. 
\end{proposition}

The non-wandering set is dense in the manifold $M$ for many dynamical systems. In the case of Anosov flows or Anosov diffeomorphisms, this density implies that the system is transitive. With this in mind, the notion of transitivity for Anosov actions by regular subcones is introduced.

\begin{definition}\label{deftransubreg}
An Anosov action is said to be transitive by subcones if, for every regular subcone $\mS$, there exists a point $x\in M$ such that the orbit $\OO_{\mS}(x)$ is dense in $M$.
\end{definition}

Note that the point $x$ depends on the regular subcone $\mS$. 
However, this is not restrictive, as every codimension 1 Anosov action is transitive by subcones. 

\begin{theorem}[Barbot - Maquera \cite{maquera-11,barbot-nils}]\label{teortransubreg} Let $\phi$ be an Anosov action of codimension $1$  of $\R^k$ on a manifold $M$ of dimension greater than $k+2$. Then for every regular subcone $\mS$ there exists $x\in M$ such that $\OO_{\mS}(x)$ is dense in $M$.  
\end{theorem}

%%%%%%%%%%%%%%%%%%%%%%%%%%%%%%%%%%%%%%%%%%%%%%%%%%%%%%%%%%%%%%
%%%%%%%%%%%%%%%%%%%%%%%%%%%%%%%%%%%%%%%%%%%%%%%%%%%%%%%%%%%%%%
%%%%%%%%%%%%%%%%%%%%%%%%%%%%%%%%%%%%%%%%%%%%%%%%%%%%%%%%%%%%%%
%%%%%%%%%%%%%%%%%%%%%%%%%%%%%%%%%%%%%%%%%%%%%%%%%%%%%%%%%%%%%%

 With all the preliminary results completed, the proof of Theorem \ref{th:dichotomy} is now addressed.

\section{Ingredients of the proof}

To enhance clarity, the proof of the main result is divided into distinct subsections, each with its own intrinsic interest. In the final subsection, these parts are combined to establish our main result, Theorem \ref{th:dichotomy}. 

\subsection{Density of stable and unstable foliation}
\label{secdensfort}

Let $\phi$ be a $\R^k$-Anosov action transitive by regular subcones on the compact connected manifold $M$. Fix a connected component $\cA_0$ of the Anosov set $\cA$, and let $\mS$ be a regular subcone.  In this context, the non-wandering set $\Omega(\mS)$ provides significant insight into the dynamics associated with the regular subcone $\mS$. The key tools for this analysis are the facts that $\mathrm{Per}(\mS)$ is dense in $\Omega(\mS)$, and that $\Omega(\mS)$ coincides with the manifold $M$ for every regular subcone $\mS$. 

The goal is to understand the density of the leaves of the weak foliations $\cF^u$ and $\cF^s$, as well as the strong foliations $\cF^{ss}$ and $\cF^{uu}$. Proposition \ref{prop:weak.foliation.density}establishes that every weak leaf is dense for any regular subcone $\mS$. For the strong leaves, Proposition \ref{Pdenso} shows that it is sufficient for the strong leaves to be dense at every point $p\in\mathrm{Per}(\mS)$.  

To study the density of the weak foliations $\cF^u$ and $\cF^s$, it is sufficient to examine a subset of each leaf. For a point $x\in M$, the following sets can be considered:

\begin{eqnarray} 
\cF^{u}_{\mS}(x) = \ds \bigcup_{v\in \mS}\phi^{v}\cF^{uu}(x)\ \mbox{and} \ 
 \cF^{s}_{\mS}(x) = \ds \bigcup_{v\in \mS}\phi^{-v}\cF^{ss}(x). 
\end{eqnarray}

The sets described are contained within the weak leaves. In particular, if $\cF^{u}_{\mS}(x)$ and $\cF^{s}_{\mS}(x)$ are dense in the manifold $M$, then the weak leaves $\cF^u(x)$ and $\cF^s(x)$ are also dense. Although the following result was established by Barbot and Maquera\cite{barbot-nils}, a simpler proof is provided here.

\begin{proposition}\label{prop:weak.foliation.density}
Let $\phi$ be an $\R^k$-Anosov action that is transitive by regular subcones. Then, for any regular subcone $\mS$ and for every point $x$, the sets  $\cF^{u}_{\mS}(x)$ and $\cF^{s}_{\mS}(x)$ are dense in $M$.
\end{proposition}

\begin{proof}
Given the transitivity of the action via regular subcones, it follows that $\Omega(\mS) = M$. For a fixed point $x\in M$ and any point $z\in \overline{\cF^{u}_{\mS}(x)}$, the local product structure, as stated in Theorem \ref{produtolocal2}, provides that for each $\delta>0$, the product neighborhood of $z\in M$ is given by
 $$ N_\delta(z)  = [\cF^{ss}_\delta(z), \cF^{u}_\delta(z)]^s.$$

Since the compact orbits are dense in $\Omega(\mS)$ by Proposition \ref{orbitacompacdensa}, there exists a compact orbit $L$ that intersects $N_\delta(z)$. By Theorem \ref{closing}, there exists $v\in \mS\cap \Gamma_{L}$ such that for all $y\in L\cap N_\delta(z)$, $\phi^v(y) = y$. Given that $y\in L\cap N_\delta(z)$, it follows that $\cF^{ss}(y)\cap\cF^{u}(z)\neq \emptyset$.  Furthermore, for all $w\in \cF^{ss}(y)\cap\cF^{u}(z)$, $\phi^{nv}(w)$ converge to $y$. 

Since $\phi^{v}\overline{\cF^{u}_{\mS}(x)}\subset\overline{\cF^{u}_{\mS}(x)}$, it follows that $L\cap N_\delta(z)$ is contained in $\overline{\cF^{u}_{\mS}(x)}$. The densities of compact orbits implies that $N_\delta(z) \subset \overline{\cF^{u}_{\mS}(x)}$, and by connectivity, $\overline{\cF^{u}_{\mS}(x)} = M$. Repeating this argument for a neighborhood $N_\delta(z) = [\cF^{s}_\delta(z),\cF^{uu}_\delta(z)]^u$ and iterating by $\phi^{-v}$, a similar reasoning establishes the density of $\cF_{\mS}^{s}(x)$. 
\end{proof}

%The density problem for strong foliations is considerably more complex than for weak foliations. Therefore, it suffices to verify that  $\cF^{uu}(p)$ and $\cF^{ss}(p)$ for all $p$ whose orbit is compact. This result is significant because, if there exists $p\in \mathrm{Per}(\mS)$ for which some strong leaf is not dense, Proposition \ref{Pfibrado} shows that there exists a transversal fibration to the action. 

Although the density problem for strong foliations is harder than for weak foliations, it suffices to verify that  $\cF^{uu}(p)$ and $\cF^{ss}(p)$ for all $p$ whose orbit is compact. This is a crucial result because, if there exists $p\in \mathrm{Per}(\mS)$ for which some strong leaf is not dense, Proposition \ref{Pfibrado} shows that there exists a transversal fibration to the action. 

\begin{proposition}\label{Pdenso} Let $\phi$ be a transitive Anosov action in regular subcones. Suppose that $\cF^{uu}(p)$  $(\cF^{ss}(p))$ is dense for every $p\in \mathrm{Per}(\mS)$, then $\cF^{uu}(x)$ $(\cF^{ss}(x))$ is dense in $M$ for all $x\in M$. 
\end{proposition}

\begin{proof} We prove the result for the strong unstable foliation and a similar argument can be used to prove the result for the stable foliation.

For an arbitrary fixed leaf $W \in \cF^{uu}$, and for $x\in M$ and $\varepsilon >0$, it is shown that $W\cap B_\varepsilon(x) \neq \emptyset$. To demonstrate this, as stated in Theorem \ref{produtolocal2}, let $N_r(x)$ denote the following open set

\begin{equation}\label{Nrviz} N_r(x) = \ds \left[\cF^{s}_r(x),\cF^{uu}_r(x)\right]^u. \end{equation}

For any $r\in(0,\delta_0),$ there exists $\delta =  \delta(r)$ such that $B_\delta (x) \subset N_r(x)$. The transitivity hypothesis implies  $\Omega(\mS) = M$, and from the density of compact orbits, a finite cover $B_{\delta(\varepsilon/4)}(p_1)$, $\ldots$, $B_{\delta(\varepsilon/4)}(p_m)$ of $M$ can be chosen, where $p_i\in \mathrm{Per}(\mS)$. Since each $\cF^{uu}(p_i)$ is dense, there exists $T>0$ such that $\cF^{uu}_{T}(p_i)\cap B_{\varepsilon/2}(x)\neq \emptyset$ for all $i$, given the finiteness of the set $\{i\}$. Choosing $v_i \in \Gamma_{p_i}\cap\mS$, and considering the dynamics properties of the foliation $\cF^{uu}$, it follows that for $n,m \in \mathbb{Z}^+$ with $n<m$, $\phi^{nv_i}\left(\cF^{uu}_{\epsilon/4}(p_i)\right)$ is contained in $\phi^{mv_i}\left(\cF^{uu}_{\epsilon/4}(p_i)\right)$ and 

$$\cF^{uu}(p_i)=\bigcup_{n\in \mathbb{Z}^+}\phi^{nv_i}\left(\cF^{uu}_{\epsilon/4}(p_i)\right). $$

Consequently, since $\{i\}$ a finite set, there exists a sufficiently large $n\in \mathbb{Z}^+$ such that, for any $i$, it follows that $\cF^{uu}_{T}(p_i) \subset \phi^{nv_i}\left(\cF^{uu}_{\epsilon/4}(p_i)\right)$ and

\begin{equation}\label{eq:1}
 \phi^{nv_i}(\cF^{uu}_{\varepsilon/4}(p_i))\cap B_{\varepsilon/2}(x) \neq \emptyset .
\end{equation}

The manifold $M$ is covered by the collection $\{B_{\delta(\varepsilon/4)}(p_i)\}_i$; therefore, for some j, it holds that $W \cap B_{\delta(\varepsilon/4)}(p_j)\neq \emptyset$, and consequently $\phi^{mv_j}W\cap B_{\delta(\epsilon/4)}(p_j)\neq \emptyset$, for any $m\in \mathbb{Z}^+$. Indeed, since $W\in \cF^{uu}$ and $B_{\delta(\epsilon/4)}(p_j)$ is contained $N_{\epsilon/4}(p_j)$, there exists $z\in \cF^{s}_{\delta(\epsilon/4)}(p_j)\cap W$. As a result, for any $m\in\mathbb{Z}^+$, 
$$d(\phi^{mv_j}(z),p_j)= d(\phi^{mv_j}(z),\phi^{mv_j}(p_j))\leq d(z,p_j)<\delta(\epsilon/4).$$

Let $n\in \mathbb{Z}^+$ be such that the intersection \ref{eq:1} is satisfied, and let $j$ be such that $W \cap B_{\delta(\varepsilon/4)}(p_j)\neq \emptyset$. 
Consider $q\in \cF^{uu}_{\varepsilon/4}(p_j) \cap \phi^{-nv_j}(B_{\varepsilon/2}(x))$. According to the definition of $N_{\epsilon/4}(p_j)$ there exists $y \in W \cap \cF^{s}_{\varepsilon/2}(q)$. 
In particular, $\mathrm{d}(\phi^{nv_j}(y),\phi^{nv_j}(q))\leq \mathrm{d}(y,q)$. Therefore, 

$$ \mathrm{d}(x,\phi^{nv_j}(y)) \leq \mathrm{d}(x,\phi^{nv_j}(q)) + \mathrm{d}(\phi^{nv_j}(q),\phi^{nv_j}(y)) < \varepsilon.$$

Since $\phi^{nv_j}(y)\in \phi^{nv_j}W$, it follows that $\phi^{nv_j}W$ is dense in $M$. Given that $\phi^{nv_j}$ is a diffeomorphism, it follows that  $W$ must also be dense in $M$.

\end{proof}

%%%%%%%%%%%%%%%%%%%%%%%%%%%%%%%%%%%%%%%%%%%%%%%%%%%%%%%%%%%%%%%%%%%%%%%%%%%%%%%%%%%%%%%%%%%%%%%%%%%%%%%%%%%%%%%%%%%%%%%%%%%%%%%%%%%%%%%%%%%%%%%%%%%%%%%%%%%%%%%%%%%%%%%%%%%%%%%%%%%%%%%%%%%%%%%%%%%%%%%%%%%%%%%%%%%%%%%%%%%%%%%%%%%%%%%%%%%%%%%%%%%%%%%%%%%%%%%%%%%%%%%%%%%%
%%%%%%%%%%%%%%%%%%%%%%%%%%%%%%%%%%%%%%%%%%%%%%%%%%%%%%%%%%%%%%%%%%%%%%%%%%%%%%%%%%%%%%%%%%%%%%%%%%%%%%%%%%%%%%%%%%%%%%%%%%%%%%%%%%%%%%

\subsection{Existence of a transversal fibration to the action}
\label{secexisfib}

The leaves of the foliations $\cF^{ss}$ and $\cF^{uu}$ are not necessarily dense in $M$. For example, if the Anosov action $\phi$ is a suspension of an Anosov action of $\Z^k$, then each leaf of $\cF^{ss}$ and $\cF^{uu}$ are contained within a compact submanifold of codimension $k$ on $M$. Moreover, in this scenario, $M$ is a fibered space, with the fibers being invariant under the action. Assuming that at least one strong leaf is not dense, then, by Proposition \ref{Pdenso}, there exists a non-dense strong leaf that passes through a compact orbit. Under these conditions, the following results are obtained for $\cF^{uu}$, and for $\cF^{ss}$, similar conclusions can be derived using analogous arguments.

\begin{proposition}\label{prop:disjointunion}
    Let $\phi$ be a transitive Anosov action in regular subcones and $p\in \mathrm{Per}(\mS)$, then $M$ is the disjoint union of $\overline{\cF^{uu}(q)}$ with $q\in \mathcal{O}_p$.
\end{proposition}

\begin{proof} Let $p\in \mathrm{Per}(\mS)$. Zorn's lemma ensures the existence of a non-empty minimal set $K \subset \overline{\cF^{uu}(p)}$ satisfying the following conditions:

\begin{enumerate}
\item $K$ is closed in $M$;
\item $K$ is $\cF^{uu}-$saturated;
%\item $\phi^{r}(K)=K$. 
\end{enumerate}

Given any $v\in \mS$, the fact $\phi$ is Anosov action implies that $\phi^{-v} (K)$ and $\phi^v (K)$ satisfy the conditions $(1)$ and $(2)$. Thus, if $K\cap \phi^v(K)\neq \emptyset$ then  $K\cap \phi^{-v}(K)\neq \emptyset$, and $(1)$ and $(2)$ hold for these intersections.  Hence, the minimality of $K$ ensures that  $K\cap \phi^v(K)=K=K\cap \phi^{-v}(K)$, which implies $K\subset \phi^v(K)$ and $K\subset \phi^{-v}(K)$. Therefore, $K\cap \phi^v(K)\neq \emptyset$ implies that $K=\phi^v(K)=\phi^{-v}(K)$.  By conditions $(1), (2)$ and Proposition \ref{prop:weak.foliation.density}, it follows that
$$M=\displaystyle \bigcup_{v\in \mS}\phi^{v}(K).$$

Note that the union above forms a partition of $M$. Specifically, $v_1, v_2 \in \mS$ are such that  $\phi^{v_1}(K)\cap\phi^{v_2}(K)\neq \emptyset$
 then $\phi^{v_1}(K)=\phi^{v_2}(K)$. In particular, if $p\in \phi^v (K)$ for some $v\in \mS$, then from conditions $(1)$ and $(2)$ it follows that $\overline{\cF^{uu}(p)} \subset\phi^v(K)$ and so $K=\phi^v(K)=\overline{\cF^{uu}(p)}$. Since the orbit of $p$ is compact and $\mathcal{O}_p=\mathcal{O}_\mS(p)$, it follows that
$$M=\displaystyle\bigcup_{q\in\mathcal{O}_p}^{\cdot}\overline{\cF^{uu}(q)}.$$
\end{proof}

The previous result does not preclude the possibility that  $M= \overline{\cF^{uu}(p)}$. If $\cF^{uu}(p)$ is not dense in $M$, then $\overline{\cF^{uu}(p)}$ is initially just a topological space and may have a complicated structure. Therefore, it is necessary to endow  $\overline{\cF^{uu}(p)}$ with a manifold structure. Consequently, due to the absence of such structure, the following proposition does not fully prepare us for the main result.

\begin{proposition}\label{Pfibrado} Let $\phi$ be an Anosov action transitive by regular subcones. Let $p\in  \mathrm{Per}(\mS)$ be such that $\cF^{uu}(p)$ is not dense in $M$. Then $M$ is fibered over a torus $\mathbb{T}^k$ with fiber $K = \overline{\cF^{uu}(p)}$,  and the action $\phi$ is a suspension of the action of $\Z^{k}$ over $K$.  
\end{proposition}

\begin{proof}
If $p\in  \mathrm{Per}(\mS)$, the isotropy subgroup $\Gamma_p$ is a lattice in $\R^k$ and $\mS$ is a regular subcone. Thus, a base $v_1, \ldots, v_k$ of $\R^k$ can be chosen such that $v_i \in \Gamma_p\cap \mS$ and $tv_i\notin \Gamma_p$, for all $t \in (0,1)$. Since the orbit of $p$ is a compact set, for each $q\in \mathcal{O}_p\setminus \{p\}$ there is only one $(t_1,\ldots, t_k)$ with $t_i\in[0,1)$ such that $$\phi(t_1v_1+\cdots+t_kv_k,p)=q.$$

The Proposition \ref{prop:disjointunion} ensures that 
$$M=\displaystyle\bigcup_{q\in\mathcal{O}_p}^{\cdot}\overline{\cF^{uu}(q)}.$$

Therefore, the projection $\pi:M\rightarrow \mathbb{T}^k$ is defined by

$$\pi\left(\phi(t_1v_1+\cdots+t_kv_k,x)\right)=\left(t_1 (\mathrm{mod} 1), 
 \ldots, t_k(\mathrm{mod} 1)\right)$$    
\end{proof}

The projection $\pi$ depends on the choice of the basis $v_1,\ldots,v_k$, and the result above is just possible because $\mathcal{O}_p$ is compact.

%%%%%%%%%%%%%%%%%%%%%%%%%%%%%%%%%%%%%%%%%%%%%%%%%%%%%%%%%%%%%%%%%%%%%%%%%%%%%%%%%%%%%%%%%%%%%%%%%%%%%%%%%%%%%%%%%%%%%%%%%%%%%%%%%%%%%%%%%%%%%%%%%%%%%%%%%%%%%%%%%%%%%%%%%%%%%%%%%%%%%%%%%%%%%%%%%%%%%%%%%%%%%%%%%%%%%%%%%%%%%%%%%%%%%%%%%%%%%%%%%%%%%%%%%%%%%%%%%%%%%%%%%%%%%%%%%%%%%%%%%%%%%%%%%%%%%%%%%%%%%%%%%%%%%%%%%%%%%%%%%%%%%%%%%%%%%%%%%%%%%%%%%%%%%%%%%%%%%%%%%%%%%%%%%%%%%%%%%%%%%%%%%%%%%%%%%%%%%%%%%

% \section{Proof of Theorem \ref{teodualidade}}
\subsection{Simultaneaous integrability}
\label{secprovdic}

To prove the main result, it is assumed that the strong foliation (stable or unstable) is not dense in $M$. In this setting, for the $\mathbb{R}^k$ Anosov action, being transitive on regular subcones is topologically conjugated to a suspension. However, the result is inconclusive because $\overline{\cF^{uu}(p)}$ lacks the structure of a manifold and is merely a topological space. To address this, in Propositions \ref{propeqint} and \ref{propdint}, a relationship is established between the densities of the strong leaves and the integrability of the subbundle $E^{ss}\oplus E^{uu}$. The proof of the main theorem is completed by demonstrating that, under our hypothesis, the fiber $K$ (as given by Proposition \ref{Pfibrado})  is a leaf tangent to the subbundle $E^{ss}\oplus E^{uu}$. However, to achieve this, it is necessary to introduce the notion of  \emph{simultaneously integrable} for Anosov actions of $\R^k$ on a manifold $M$, following the ideas presented by Plante \cite{plante}. 

Let $N =  [\cF^{s}_\delta (x), \cF^{uu}_\delta (x)]^u$ be a product neighborhood, as Theorem \ref{produtolocal2}, of the point $x \in M$. If $y$ and $z$ are in the same strong stable leaf within $N$, then there exist $\rho>0$ such that the map $\pi_{y,z}: \cF^{s}_\rho(y) \rightarrow \cF^{s}_\delta(z)$, given by the projection along the unstable leaves, is well-defined.

\begin{definition}
\emph{The main foliations $\cF^{uu}$ and $\cF^{ss}$ are called} simultaneously integrable in the product neighborhood $N$ \emph{if for a given $y,z$ and $\rho$ as above}
\begin{equation*}\pi_{y,z}(\cF^{ss}(u)\cap\cF^{s}_\rho(y)) \subset \cF^{ss}(\pi_{y,z}(u))\cap \cF^{s}_\delta(z), \ \ \forall u\in \cF^{s}_\rho(y).\end{equation*} 
\emph{The foliations $\cF^{uu}$ and $\cF^{ss}$ are \emph{simultaneously integrable} if all point in $M$ lie in a product neighborhood $N$  where the foliations are simultaneously integrable. In this case, the subbundles $E^{ss}$ and $E^{uu}$ are called simultaneously integrable.}
\end{definition}

The goal is to relate the notion of simultaneous integrability with the concept of integrability for foliations. Recall that a subbundle $E$ is said to be integrable if there is a foliation with leaves tangent to $E$. The following result ensures that this definition is equivalent to the integrability of $E^{ss}\oplus E^{uu}$.

\begin{proposition}
\label{propeqint} Let $\phi$ be a transitive Anosov action by regular subcones.
$\cF^{uu}$ and $\cF^{ss}$ are simultaneously integrable if, and only if, $E^{ss}\oplus E^{uu}$ is integrable.
\end{proposition}

\begin{proof}
If $E^{ss}\oplus E^{uu}$ is integrable, then $\cF^{ss}$ and $\cF^{uu}$ are immediately simultaneously integrable. Now, assume that $\cF^{ss}$ and $\cF^{uu}$ are simultaneously integrable. Let $p\in M$ and  $y\in \cF^{uu}_{\delta/2}(p)$. Then the projection $\pi_{p,y}(p) = y$, and for any point $w\in \cF^{s}_{\delta/2}(p)$, the projection satisfies 
$$\pi_{p,y}(w) = \cF^{uu}_{\delta}(w)\cap \cF^{s}_{\delta}(y) = [w,y]^s.$$
Moreover, $[\cF^{ss}_{\delta}(p),y]^s=\pi_{p,y}(\cF^{ss}_{\delta/2}(p)) \subset \cF^{ss}_{\delta}(y)$, and the discs
$$ V(p,\delta/2) =[\cF^{ss}_{\delta/2}(p), \cF^{uu}_{\delta/2}(p)]^s = \ds \bigcup_{y\in\cF^{uu}_{\delta/2}(p)}^{\cdot}[\cF^{ss}_{\delta/2}(p),y]^s\subset \ds \bigcup_{y\in\cF^{uu}_{\delta/2}(p)}^{\cdot} \cF^{ss}_{\delta}(y).$$

Notice that the discs $V(p,\delta/2)$ are invariant under the regular subcone  $\mS$. This means that for every $v\in \mS $, there exists positive constants $\alpha$ and $\beta$ such that $\phi^{v}V(p,\alpha)$ is contained in $V(\phi^{v}(p),\beta)$. Additionally, within the neighborhood $N$, it can be observed that $L(p) = \ds \bigcup_{y\in \cF^{uu}(p)}\cF^{ss}(y)$ forms a submanifold of $M$ with codimension $k$. 

Let $\cF$ denote the collection of submanifolds described above, where the tangent bundle is $E^{ss}\oplus E^{uu}$. We will show below that  $\cF$  is of $C^1$ class. Considering the invariance of the subcone $\mS$,  a coordinate system $\psi: U \rightarrow \R^{k+n}$ ($k+n= \mathrm{dim}M$) can be defined for each point $p\in \mathrm{Per}(\mS)$ in the following manner.

The isotropy subgroup $\Gamma_p$ is a lattice on $\R^k$, allowing for the selection of a basis $\{u_1,\ldots,u_k\}$ of $\R^{k}$ in $\mS\cap \Gamma_p$. 
Let $B\subset L(p)$ be an open ball within the leaf  $L(p)$ of $\cF$ that contains the point $p$, and let $\eta:B\rightarrow \R^n$ be a $C^1$ embedding. Define $B^{k}_{\delta}$ the ball in $\R^k$ centered at $(1,0,\ldots,0)$ with radius $\delta$. Consider the set
 $$ U = \ds \bigcup_{(t_1,\ldots,t_k)\in B^{k}_\delta}\phi(t_1u_1+\cdots +t_ku_k,B)$$ 
 where if $\delta>0$ is sufficiently small, the map $B\times B^{k}_\delta \rightarrow U$ given by $(x,(t_1,\ldots,t_k))$ associated to $\phi(t_1u_1+\cdots +t_ku_k,x)$ is a homeomorphism and the linear combination $t_1u_1+\cdots+t_ku_k\in \mS$. Define the coordinate map $$\psi: U \rightarrow \R^{k+n},\ \psi(\phi(t_1u_1+\cdots +t_ku_k,x))=(\eta(x),(t_1,\ldots,t_k)).$$ 

The map $\psi$ is $C^1-$class, and since $\mathrm{Per}(\mS)$ is dense in $M$, the collections of charts $\{\psi\}$ determines a $C^1$ foliation $\cF$ of $M$. 
\end{proof}

Up to now, we have established that the density of strong leaves can be studied only at points $p \in \mathrm{Per}(\mS)$, Proposition \ref{Pdenso}. If some strong leave is not dense, then $\phi$ is topologically conjugated to a suspension of the $\Z^k$-action over $K$, Proposition \ref{Pfibrado}. To develop a differential structure for the fiber $K$, we need to relate the density of the strong leaves to the integrability of $E^{ss}\oplus E^{uu}$.

\begin{proposition}\label{propdint}
If a strong unstable leaf or strong stable leaf is not dense in $M$, then $\cF^{uu}$ and $\cF^{ss}$ are simultaneously integrable.
\end{proposition}

\begin{proof} Assume that some leaf $\cF^{uu}$ is not dense. By Proposition \ref{Pdenso} there exists a $\cF^{uu}(p)$ that is not dense in $M$ for some point $p\in \mathrm{Per}(\mS)$. Denote by $\Gamma=\Gamma_{\mathcal{O}_p}$ the isotropy subgroup.  Let $K= \overline{\cF^{uu}(p)}$, the fiber that saturates the foliation $\cF^{uu}$, as stated in Proposition \ref{Pfibrado}. Notice that $K$ is also saturated by $\cF^{ss}$. Indeed, given $y$ and $z$ points in the same strong stable leaf, for any $v_i\in \Gamma\cap \mS $ as per Proposition \ref{Pfibrado}, it holds that

$$\ds \lim_{n\rightarrow \infty}\mathrm{d}(\phi^{nv_i}(y),\phi^{nv_i}(z)) = 0, \ \ i=1,\ldots, k.$$

Since each $\phi^{nv_i}$, $n\in \mathbb{Z}$, leaves invariant each fiber, it follows that $y$ and $z$ must lie in the same fiber. In other words, the fibers are saturated by $\cF^{ss}$. Consider the map $\pi_{y,z}: \cF^{s}_\rho(y)\rightarrow \cF^{s}_\delta(z)$ as previously defined, where $y$ and $z$ are in the same unstable product neighborhood. Since $\pi_{y,z}$ is the projection along the foliation $\cF^{uu}$, the sets $\pi_{y,z}(\cF^{ss}(y)\cap \cF^{s}_\rho(y))$ and $\cF^{ss}(y)\cap \cF^{s}_\rho(y)$ remain within the same fiber, as the fibers are saturated.

Assume that $\pi_{y,z}(\cF^{ss}(y)\cap \cF^{s}_\rho(y))$ is not contained in $\cF^{ss}(z) \cap \cF^{s}_\delta(z)$. Then there exists a neighborhood $U$ of $0$ in $\R^k$ such that for all $t\in U$. it follows that $\phi(t_1v_1+\cdots+t_kv_k,K)\cap K\neq \emptyset$. This situation contradicts Proposition \ref{Pfibrado}; therefore, the result follows. 
\end{proof}

We are now ready to conclude the dichotomy between the densities of strong stable and unstable leaves and the action being a suspension of a $\Z^k$ action.

\subsection{Proof of Theorem \ref{th:dichotomy}}

 If any unstable or stable leaf is not dense in $M$, then $E^{uu}\oplus E^{ss}$ is integrable, as established by Propositions \ref{propeqint} and \ref{propdint}. Let $L$ be a leaf of the foliation $\cF$ that integrates $E^{uu}\oplus E^{ss}$. Given that $K$, as determined in Proposition \ref{Pfibrado}, is $\cF^{uu}-$saturated and, as noted in Proposition \ref{propdint}, is also $\cF^{ss}-$saturated, $K$ can be chosen so that the leave $L$ is contained within $K$. Since $L$ includes a strong unstable leaf, it follows that $L$ is dense in $K$.

Suppose that $L$ is not closed and let $z\in \overline{L}$. For a given $\beta>0$, define the  cross-section $\Sigma(z) = [\cF^{uu}_\beta(z), \cF^{ss}_\beta(z)]^s$. Using this definition and $v_i$ as determined in Proposition \ref{Pfibrado}, the following neighborhood can be constructed: $$V(z) = \ds \bigcup_{t_1 = 0}^1 \cdots \bigcup_{t_k=0}^1 \phi(t_1v_1+\cdots+t_kv_k, \Sigma(z)).$$

Thus, there exist points $p$ and $q$ in $L\cap V(z)$ that are arbitrarily close in the metric of $M$, yet arbitrarily distant in the induced metric of $L$. Given that $E^{ss}$ and $E^{uu}$ are simultaneously integrable, there exists $\alpha>0$ such that the cross-sections  $\Sigma(p) = [\cF^{uu}_\alpha(p), \cF^{ss}_\alpha(p)]^s$ and $\Sigma(q) = [\cF^{uu}_\alpha(q), \cF^{ss}_\alpha(q)]^s$ are contained in $L\cap V(z)$. 

This leads to a contradiction, as it implies the existence of $(t_1,\cdots, t_k)$ with arbitrarily small norm such that $\phi(t_1v_1+\cdots+t_kv_k, K)\cap K \neq \emptyset$. Consequently, $L$ must be closed, leading to the conclusion that $L= K$. This completes the proof of our main result.

 \textit{Acknowledgements.} R.R.L. was supported by CNPq grant 165832/2014-2, C.M. was supported by Fundação de Amparo à Pesquisa do Estado de São Paulo (FAPESP) (grant 2022/16455-6),  R.V. was partially supported by Conselho Nacional de Desenvolvimento Cient\'ifico e Tecnol\'ogico (CNPq) (grants 313947/2020-1 and 314978/2023-2), and partially supported by Fundação de Amparo à Pesquisa do Estado de São Paulo (FAPESP) (grants 2016/22475-9, 17/06463-3 and 18/13481-0).

\bibliographystyle{plain}
\bibliography{references}
\end{document}